\documentclass[review,3p,11pt]{elsarticle}

\usepackage{lineno,hyperref}
\modulolinenumbers[5]

\makeatletter
\def\ps@pprintTitle{%
  \let\@oddhead\@empty
  \let\@evenhead\@empty
  \def\@oddfoot{\reset@font\hfil\thepage\hfil}
  \let\@evenfoot\@oddfoot
}
\makeatother
\makeatother

\usepackage{amsmath,amsthm,amssymb,amsfonts,amscd,url}
\usepackage{mathrsfs}
\usepackage{bm}
\usepackage{mathscinet}
\usepackage{enumerate}
\usepackage{verbatim}
\usepackage[T1]{fontenc}
\usepackage[dvipsnames]{xcolor}

\newtheorem{thm}{Theorem}

\newtheorem{lem}[thm]{Lemma}

\newtheorem*{cor*}{Corollary}
\newtheorem{prop}[thm]{Proposition}

\newtheorem*{thm*}{Theorem}

\def\R{{\mathbb R}}

\def\Lop{{\mathcal L}}
\def\d{{\,\rm d}}

\def\e{{\rm e}}

\def\la{\left\langle}
\def\ra{\right\rangle}
\def\ve{{\varepsilon}}
\def\Id{{\bf{Id}}}
\def\A{{\bf{A}}}

\def\Xint#1{\mathchoice
{\XXint\displaystyle\textstyle{#1}}%
{\XXint\textstyle\scriptstyle{#1}}%
{\XXint\scriptstyle\scriptscriptstyle{#1}}%
{\XXint\scriptscriptstyle\scriptscriptstyle{#1}}%
\!\int}
\def\XXint#1#2#3{{\setbox0=\hbox{$#1{#2#3}{\int}$ }
\vcenter{\hbox{$#2#3$ }}\kern-.58\wd0}}

\def\dashint{\Xint-}

\begin{document}

\begin{frontmatter}

\title{Optimal Liouville theorem for a semilinear Ornstein-Uhlenbeck equation}

\author{Micha\l{} Fabisiak}
\author{Miko\l{}aj Sier\.z\k{e}ga\corref{mycorrespondingauthor}}

\address{Faculty of Mathematics, Informatics and Mechanics,  University of Warsaw\\  Banacha 2, 02-097 Warsaw, Poland.}

\cortext[mycorrespondingauthor]{Corresponding author}
\ead{m.sierzega@uw.edu.pl}

\begin{abstract}
The question of triviality of solutions of the semilinear Ornstein-Uhlenbeck equation,
\[
\Delta w-\frac{1}{2} \langle  x,\nabla w\rangle-\frac{\lambda}{p-1}w+|w|^{p-1}w=0,
\]
is considered. It is shown, that if $p>1$ is Sobolev subcritical or critical and $\lambda\leq 1$, then all bounded entire solutions are constant. Moreover, in the critical case, the same conclusion holds in the subclass of radial solutions provided that $n\geq 4$ and $\lambda \in \left[\frac{3 n}{2(n-1)},2\right]$. 
\end{abstract}

\begin{keyword}
Ornstein-Uhlenbeck operator\sep semilinear elliptic equation\sep Liouville theorem \sep Rellich-Pohozaev identity
\end{keyword}

\end{frontmatter}




\section{Introduction}
In this paper we consider the question of triviality of entire solutions to the semilinear  problem 
\begin{equation}\label{eq:eq}
\Delta w-\frac{1}{2} \langle x,\nabla w\rangle -\frac{\lambda}{p-1}w+|w|^{p-1}w=0\quad\mbox{on}\quad \R^n. 
\end{equation}
Here $\langle \cdot,\cdot\rangle$ is the standard inner product in $\R^n$, $\lambda$ is a scalar and $1<p\leq p_S$, where $p_S>0$ is known as the Sobolev exponent and is given by  
\[
p_S:=\begin{cases}+\infty\quad\mbox{when}\quad n=1\mbox{\,or\,}\, 2,\\
\frac{n+2}{n-2}\quad \mbox{otherwise}.\end{cases}
\]
By an entire  solution we understand a $C^2(\R^n)$ function that satisfies \eqref{eq:eq} pointwise in $\R^n$. We are interested in a Liouville-type result whereby bounded entire solutions are shown to be necessarily constant. 

Equation \eqref{eq:eq} arises as a natural generalisation of an important special case, $\lambda=1$, which plays a key r\^ole in the analysis of the self-similar singularity formation for the much studied semilinear heat flow,
\begin{equation}\label{eq:fuj}
u_t-\Delta u=|u|^{p-1}u \quad\mbox{on}\quad \R^n,
\end{equation}
see e.g.\,\cite{FILIPPASKOHN,GIGAKOHN1985,GIGAKOHN1987,GIGAKOHN1989}. More precisely, one may gain a sharper picture of the blow-up process by rephrasing \eqref{eq:fuj} in the so-called \emph{similarity variables} which incorporate the scaling properties of the equation. Then \eqref{eq:eq} describes the ground states of the transformed equation which in turn may be interpreted as \emph{admissible local profiles} for potential singularities of \eqref{eq:fuj}. The central result in this area - the Liouville theorem in \cite{GIGAKOHN1985} - states:
\begin{center}
{\it  When $1<p\leq p_S$ and $\lambda=1$, then  all bounded entire solutions of \eqref{eq:eq} are constant.}
 \end{center}
It is also known that nontrivial bounded solutions do exist when $p>p_S$ \cite{BUDDQI1989} and so the above result is optimal.
For the purpose of our discussion we might call the result of Giga and Kohn \emph{unconditional} as smoothness and boundedness of solutions are the classical prerequisites for Liouville-type theorems.   

In \cite{GIGA1986} Giga showed that when $\lambda>1$ and $1<p<p_S$  then bounded nonconstant solutions do exist. One would then expect triviality of bounded solutions for $\lambda<1$, i.e.\,that the following Liouville theorem holds: 
\begin{center}\label{claim}
{\it When $1<p\leq p_S$ and $\lambda\leq 1$, then all bounded entire solutions of \eqref{eq:eq} are constant. }
\end{center}
Plausible as it is, this problem turns out to be quite subtle and, to the best of our knowledge, 
only two \emph{conditional} results are available. 
\subsection{Theorem 3 and Proposition 9 in \cite{GIGA1986}}
\emph{Let $0\leq \lambda\leq 1$ and $1<p\leq p_S$. Then all entire  solutions of \eqref{eq:eq} that are bounded, positive, radial and nonincreasing are constant. 
}

\noindent In fact, more can be said as  the proof of the above applies to a more general class of solutions that are smooth, bounded, such that $|\nabla w|$ grows at most polynomially in $|x|$ and display a kind of monotonicity under scaling expressed by
\begin{equation}\label{giga cond}
\frac{\d}{\d \ve}\bigg|_{\ve=1}\int_{\R^n}|w(\ve x)|^2\e^{-\frac{|x|^2}{4}}\d x\leq 0.
\end{equation}
Second result, due to \cite{PELETIER1986},  requires sufficiently rapid vanishing of $|w|$ and $|\nabla w|$ at space infinity and moreover restricts the range of $\lambda$ and $p$. 

\subsection{Discussion on p.\,97 in \cite{PELETIER1986}}
\emph{Let $0\leq \lambda< \frac{n(p-1)}{4}$ and $1<p<p_S$, then all entire solutions of \eqref{eq:eq}  that are bounded and belong to $ H^1(\R^n)$ are constant. }

Our aim in this work is to demonstrate that the unconditional Liouville theorem indeed holds. Notice, that a special family of regimes falls outside of the scope of the above-mentioned results. Namely the Sobolev critical case $p=p_S$ with $\lambda>1$. The methodology developed below yields results in this case as well.

\subsection{The Liouville Theorem}
\emph{Let $1<p\leq p_S$ and $\lambda\leq 1$, then every bounded entire solution of \eqref{eq:eq} is constant. If moreover we consider radial solutions,  then the same conclusion also holds for $\lambda\in \left[\frac{3n}{2(n-1)},2\right]$ provided that $n\geq 4$ and  $ p=p_S$.}

While in this line of research the drift term $\langle  x,\nabla w\rangle$ appears as a byproduct of rephrasing the Fujita equation in the similarity variables, the differential operator 
 \begin{equation}\label{L}
 \mathcal L:=\Delta -\frac{1}{2}\langle x, \nabla\rangle =\frac{1}{\rho}\nabla \cdot(\rho\nabla )\quad  \mbox{with }\quad \rho(x)=\e^{-\frac{|x|^2}{4}} 
 \end{equation}
is well-known and widely used in the area of stochastic analysis, where it goes by the name of the Ornstein-Uhlenbeck operator, see e.g.\,\cite{SJOGREN}. In particular, triviality of smooth  bounded solutions for the linear problem $\mathcal Lu=0$ may be found in \cite{PRIOLAZABCZYK2004}. That result is cast in the probabilistic language. However, purely analytical techniques have been recently employed to further sharpen the sufficient conditions for the triviality of solutions \cite{KOGOJLANCONELLIPRIOLA2020}. 

Herein we study a Liouville-type theorem for a \emph{semilinear} problem and our methodology is purely analytical but different from the one employed in \cite{KOGOJLANCONELLIPRIOLA2020}.

The purpose of this paper is twofold. We want to resolve an intriguing open problem highlighted by Giga in \cite{GIGA1986} but we also aim to bring home a methodological point. The Rellich-Pohozaev type of argument often used in Liouville-type theorems may be briefly described as an instance of subjecting the problem under consideration to\emph{ just the right kind of test functions} as to obtain decisive integral identities that exclude nontrivial solutions. In the classical works of Rellich \cite{RELLICH1940} on the eigenproblem for the Laplace operator and Pohozaev \cite{Pohozaev1965} on its nonlinear extension the backbone of the technique is the use of the derivative of the solution in the direction of the Euler field $x$ as a test function. It is an object that displays some particularly useful computational properties under integration by parts. In particular, it is responsible for the occurrence of the dimension of the underlying domain in calculations through the relation $\nabla \cdot x=n$. The relationship between the dimension of space, the Euler field $x$ and the divergence theorem lent itself to far-reaching generalisations that include among others higher order equations and systems \cite{PUCCISERRIN1986}, Riemannian manifolds setting \cite{GOVERORSTED2013} or fractional powers of the Laplace operator \cite{OTONSERRA2014}. An analogous technique (dubbed {\it Friedrichs' a,b,c-method}) has been employed by Protter \cite{PROTTER1954} and Morawetz \cite{MORAWETZ1961} in the context of wave equations. 

For simple second-order equations posed in $\R^n$ the idea to use the Euler field is fairly evident. There are however important instances when this choice fails to deliver and a different proposal is required. One example of such a situation is the case of the Brezis-Nirenberg problem in three dimensions \cite{BREZISNIRENBERG1983}. Instead of relying on guesswork it is possible to derive a suitable test function in accordance with the properties we need it to display (see Lemma 1.4 in the proof of Theorem 1.2 in \cite{BREZISNIRENBERG1983}). In this work we want to describe in detail a procedure of deriving the 'right kind of test functions' in the context of semilinear Ornstein-Uhlenbeck equations. 

The article is organised as follows. In the Preliminaries we fix notation conventions, review the classical Liouville theorem of Giga and Kohn, extract some key observations from their technique and outline a plan for its generalisation. Next, we move on to the proof of our Liouville theorem. The demonstration relies on a collection of more technical propositions which are subsequently justified. In the Discussion we provide some remarks concerning the wider picture and point some further directions of investigation.  Lastly, we provide an Appendix that collects a brief selection of facts concerning Kummer functions that are central to the whole scheme.

\section{Preliminaries}\label{LT}
\subsection{Notational conventions}
We will work with vector fields and derived objects that depend on the parameter $\lambda$. We will however omit the $\lambda$ subscript whenever it seems superfluous. In particular, we choose to always call a solution of \eqref{eq:eq} $w$ instead of $w_\lambda$.   

When dealing with a radial function, i.e.\,one that depends on the argument through its modulus $r=|x|$ (like for example $\rho(x)=\e^{-|x|^2/4}$), we will  identify $\rho(x)\equiv\rho(r)$. Thus, with this convention we may write $\nabla \rho=\rho'\nu$, where $\nu=x/r$.

We will say that a matrix field $\A:\R^n\mapsto M^{n\times n}(\R)$ is positive (or negative)  definite on $\Omega\subseteq\R^n$ if $\A(x)$ is a positive (or negative) definite matrix for every $x\in \Omega$.  

In order to make some calculations easier to follow we will use the standard notation for the inner product in $\R^n$ also to denote the action of a symmetric matrix on a vector. Thus given a symmetric matrix $\mathbf{M}\in M^{n\times n}(\R)$ and a vector $\beta\in \R^n$ we may write $\la \mathbf{M},\beta\ra$ in place of the more common $\mathbf{M}\beta$. 

We stress that the operator $\nabla$ will be applied both to functions and vectors alike to produce gradient vectors and Jacobian matrices respectively. In particular, the hessian matrix of a smooth function $h$ will be given by $\nabla(\nabla h)=\nabla^2 h $, where the $\nabla$ operator appears in both contexts.

\subsection*{The extended Giga \& Kohn scheme}
Our reasoning is inspired by Proposition 2 and Theorem 1 in the seminal paper of Giga and Kohn \cite{GIGAKOHN1985} and we believe that a brief sketch of their original argument will render our generalisation more transparent. 
Consider then $\lambda=1 $ and $w$, a bounded entire solution of equation \eqref{eq:eq}. The equation is tested over $\R^n$ with three specific multipliers: $w\rho$, $|x|^2w\rho$ and $\langle \nabla w,x\rangle \rho$. In effect we obtain the following integral identities:
\begin{equation}\label{GK1}\tag{A}
0=\int_{\R^n}|\nabla w|^2\rho+\frac{1}{p-1}\int_{\R^n}|w|^2\rho-\int_{\R^n}|w|^{p+1}\rho,
\end{equation}
\begin{equation}\label{GK2}\tag{B}
0=\int_{\R^n}|x|^2|\nabla w|^2\rho+\left(\frac{1}{2}+\frac{1}{p-1}\right)\int_{\R^n}|x|^2|w|^2\rho-n\int_{\R^n}|w|^2\rho-\int_{\R^n}|x|^2|w|^{p+1}\rho
\end{equation}
and 
\begin{equation}\label{GK3}\tag{C}\begin{split}
0=&-\frac{n-2}{2}\int_{\R^n}|\nabla w|^2\rho-\frac{n}{2(p-1)}\int_{\R^n}|w|^2\rho+\frac{n}{p+1}\int_{\R^n}|w|^{p+1}\rho\nonumber \\
&+\frac{1}{4}\int_{\R^n}|x|^2|\nabla w|^2\rho+\frac{1}{4(p-1)}\int_{\R^n}|x|^2|w|^2\rho-\frac{1}{2(p+1)}\int_{\R^n}|x|^2|w|^{p+1}\rho.\end{split}
\end{equation}
These identities are then combined in such a way, as to cancel out integrals that involve $|w|^2$ and $|w|^{p+1}$ in their integrands:
\[
\frac{n}{p+1}\eqref{GK1}-\frac{1}{2(p+1)}\eqref{GK2}+\eqref{GK3}=0.
\]
The resulting single identity employs integrals that depend on the solution $w$ via $|\nabla w|^2$ only: 
\begin{equation}\label{eq:123}
\left(\frac{n}{p+1}-\frac{n-2}{2}\right)\int_{\R^n}|\nabla w|^2\rho+\frac{1}{2}\left(\frac{1}{2}-\frac{1}{p+1}\right)\int_{\R^n}|x|^2|\nabla w|^2\rho=0.
\end{equation} 
Since $p>1$, the the second coefficient is always positive while the first one is nonnegative whenever $p\leq p_S$. Hence, we conclude that for $p$ within this range $\nabla w\equiv 0$ on $\R^n$ and in consequence $w$ has to be constant. 

To make the procedure fully rigorous it remains to justify that all integrals involved are well defined. By working with equation \eqref{eq:fuj} it may be argued that all first and second partial derivatives of $w$ are uniformly bounded on $\R^n$, see Proposition 1 and Proposition 1' in \cite{GIGAKOHN1985}. With these bounds provided all calculations may be performed on balls $B_r$ of finite size with 
\begin{equation}\label{eq:1234}
\left(\frac{n}{p+1}-\frac{n-2}{2}\right)\int_{B_r}|\nabla w|^2\rho+\frac{1}{2}\left(\frac{1}{2}-\frac{1}{p+1}\right)\int_{B_r}|x|^2|\nabla w|^2\rho=\int_{\partial B_r}\eta
\end{equation}
in place of \eqref{eq:123} with all the byproducts of integration by parts accounted for in $\eta$. Finally the radius may  be chosen arbitrarily since the exponentially decaying weight $\rho$ easily dominates all boundary integrands.

An interesting aspect of this procedure is that, unlike the great majority of calculations in this area of mathematics, it operates on identities rather than estimates. The upside is that a fine-tuned, equation-specific computation yields an optimal result in a span of several lines of standard if slightly tedious calculation. The downside is that this procedure is extremely rigid. 

An attempt to apply the same argument to the case of $\lambda<1$ turns out to be only partially successful \cite{GIGA1986}. In essence, it is not possible anymore to cancel out the $|w|^2$ and $|w|^{p+1}$ integrals at the same time and the counterpart of identity \eqref{eq:123} now reads 
\begin{align*}
\left(\frac{n}{p+1}-\frac{n-2}{2}\right)\int_{\R^n}\left|\nabla w\right|^2\rho+&\frac{1}{2}\left(\frac{1}{2}-\frac{1}{p+1}\right)\int_{\R^n}|x|^2\left|\nabla w\right|^2\rho\\
&=\frac{\lambda-1}{2(p+1)}\left(n\int_{\R^n}\left|w\right|^2\rho-\frac{1}{2}\int_{\R^n}|x|^2\left| w\right|^2\rho\right)\\
&=-\frac{\lambda-1}{2(p+1)}\frac{\d}{\d \ve}\bigg|_{\ve=1}\int_{\R^n}|w(\ve x)|^2
\rho\d x.
\end{align*} 
Thus, once $\lambda$ deviates from the critical value, cancelations appearing in the original argument are not as thorough as needed and we are left with troublesome terms that force additional assumptions on the solution. Also, when $\lambda\neq 1$, some extra care must be taken as the boundedness of derivatives of $w$ -- asserted in the said Proposition 1' -- does not automatically extend to these cases. This drawback stems from the fact that bounds derived in the above-mentioned Proposition 1 (that concerns the Fujita equation \eqref{eq:fuj}) translate to bounds for solutions of equation \eqref{eq:eq} (with $\lambda=1$) in Proposition 1' via their similarity-scaling connection. In our case this connection is lost and we will provide necessary bounds via a different route.      

It turns out therefore that for $\lambda\neq 1$ the multipliers -- $w\rho$, $|x|^2w\rho$ and $\langle \nabla w,x\rangle \rho$ -- are not as efficient in producing decisive integral identities as required for the triviality argument. A natural course of action would be then to modify this technique by introducing other multipliers that display for $ \lambda< 1$ the beneficial properties encountered in the original method. In order to do this we first extract some key lessons from  the argument of Giga and Kohn:
\begin{itemize}
\item the identity \eqref{eq:123} may be rephrased as:
\[
\int_{\R^n}\left\langle \nabla w, \A\nabla w\right\rangle=0\quad \mbox{with}\quad \A(x)=\left[\left(\frac{n}{p+1}-\frac{n-2}{2}\right)+\frac{1}{2}\left(\frac{1}{2}-\frac{1}{p+1}\right)|x|^2\right]\rho\Id,
 \]
so that the triviality of $w$ stems from $\A(x)$ being positive definite outside the origin (in fact $\A(x)$ is positive definite on $\R^n$ when $p<p_S$ but due to the continuity of $w$ it suffices to ensure this property on $\dot \R^n=\R^n\setminus\{0\}$),
\item also, \eqref{eq:123} may be derived by using two multipliers only: 
\[
\langle \nabla w,x\rangle \rho \quad \mbox{and}\quad  \left(n-\frac{|x|^2}{2}\right)w\rho ,
\]
\item these multipliers are related in a very particular way, i.e.\,if we set the vector field $\Phi:=\nabla \rho$, then 
\[
\langle \nabla w,x\rangle \rho=-2 \langle \nabla w,\Phi\rangle \quad \mbox{and}\quad \left(n-\frac{|x|^2}{2}\right)w\rho = -2w\nabla \cdot \Phi
\]
(as we will see below the appearance of an Hermite-type polynomial, $n-\frac{|x|^2}{2}$, in the above test function is no coincidence),
\item for $\lambda=1$ these two multipliers made it possible to cancel out \emph{all 'non-gradient'} integrals, but when $\lambda\neq 1$ the conditions needed for cancellation of $|w|^2$ and $|w|^{p+1}$ integrands diverge,  
\item the cancellation of the $|w|^{p+1}$ integrals results from integration by parts and linear combination of equations and does not depend on the actual form of the vector field $\Phi$. This is \emph{not the case} for $|w|^2$ integrals, where particular properties of $\Phi=\nabla\rho$ are required.      
\end{itemize}

These observations suggest a clear scheme of extending the Rellich-Pohozaev argument of Giga and Kohn to $\lambda<1$. First, we will derive two integral identities by testing \eqref{eq:eq} on finite size balls $B_r$ with multipliers of the form $\la \nabla w,\Phi_\lambda\ra$ and $w\nabla \cdot \Phi_\lambda$ where the vector field $\Phi_\lambda$ is not as yet specified and combine them into one identity that does not contain $|w|^{p+1}$ integrals. All gradient terms in the resulting identity should be rephrased in terms of a bilinear form $\la \nabla w,\A_\lambda\nabla w\ra$ for some matrix field $\A_\lambda$ associated with $\Phi_\lambda$. This step is covered in Proposition \ref{prop:main1} below. Second, we will derive an actual form of the vector field that results in vanishing of $|w|^2$ integrals - covered in Proposition \ref{prop:main2}. Third, we will check under what conditions our matrix field $\A_\lambda$ is positive definite a.e.\,in $\R^n$ - see Proposition \ref{A+}. Finally, we will investigate the asymptotic behaviour of the collected boundary terms left after repeated applications of the divergence theorem - Proposition \ref{prop:main4}. At the intersection of these four steps we will find our generalisation.

\section{The Liouville theorem}
We begin with a statement of four technical propositions, proofs of which are postponed until the next section.  

\begin{prop}\label{prop:main1}
Let $\Omega\subset \R^n$ be a smooth bounded domain with $\nu$ an external normal vector field to its boundary and let $\Phi$ be a smooth vector field on $\R^n$. Define a matrix field

\begin{equation}\label{eq:A}
\A:=\rho\nabla \frac{\Phi}{\rho}-\left(\frac{1}{2}-\frac{1}{p+1}\right)\nabla \cdot \Phi \Id, \\
\end{equation}
a vector field
\begin{equation}
a:=\frac{1}{2(p+1)}\left(\rho\nabla\frac{\nabla \cdot\Phi}{\rho}+\lambda\Phi\right)
\end{equation}
and a scalar function 
\begin{equation}\label{eta}
 \begin{split}
 \eta:=&\langle \nabla w,\nu\rangle\langle \nabla w,\Phi\rangle-\frac{1}{2}|\nabla w|^2\langle \Phi,\nu\rangle+\frac{1}{p+1}\la\nabla w,\nu\ra w\nabla \cdot\Phi\\
&-|w|^2\la a,\nu\ra+\frac{1}{p+1}\left(-\frac{\lambda}{p-1}|w|^2+|w|^{p+1}\right)\langle\Phi,\nu\rangle. 
 \end{split}
 \end{equation}
Suppose $w$ is a smooth bounded solution of \eqref{eq:eq}, then

\begin{equation}\label{id}
\int_{\Omega} \la\nabla w, \A\nabla w\ra=\int_{\Omega}|w|^2\nabla \cdot a+\int_{\partial \Omega}\eta.
\end{equation}
\end{prop}

\begin{prop}\label{prop:main2}
For every $\mu\in \R$ we may find a smooth vector field $\Phi_\mu$ such that:
\begin{equation}\label{phi}
\nabla \cdot\left(\rho\nabla\frac{\nabla \cdot\Phi_\mu}{\rho}+\mu\Phi_\mu\right)=0 \quad\mbox{  on }\quad\R^n.
\end{equation}
\end{prop}

\begin{prop}\label{A+}
Let  $\Phi_\mu$ be the vector field derived in the proof of Proposition \ref{prop:main2}. Then the matrix field $\A_\mu$ as defined by \eqref{eq:A} is: 

\begin{enumerate}[a)]
\item  positive definite on $\dot\R^n$ when $\mu\in\left[-\frac{n}{2},1\right]$,
\item  negative definite on $\dot\R^n$ when $p=p_S$, $n\geq 4$ and $\mu \in \left[\frac{3n}{2(n-1)},2\right]$,
\item neither positive nor negative semidefinite on $\R^n$ otherwise. 
\end{enumerate}
\end{prop}

\begin{prop}\label{prop:main4}
Let $\lambda\geq 0$, $\Phi_\lambda$ be the associated vector field derived in the proof of Proposition \ref{prop:main2}, $w$ be a bounded entire solution of \eqref{eq:eq} and $\eta_\lambda$ be defined via \eqref{eta}. Then, there exists a monotone sequence $r_k$ of radii such that $ \int_{\partial B_{r_k}}\eta_{\lambda}\to 0$ as $r_k\to \infty$.
\end{prop}


We are now ready to prove our main result. 
\begin{proof}[Proof of the Liouville Theorem]
The argument is a straightforward consequence of the propositions above with some details depending on the sign of $\lambda$. 

First, fix $\lambda\leq 1$ and apply Proposition \ref{prop:main2} with $\mu=\lambda$ to define the vector field $\Phi$ according to 
 \[
 \Phi:=\begin{cases}\Phi_\lambda\quad \mbox{ for }\quad \lambda\geq 0,\\
 \Phi_0\quad \mbox{ for }\quad \lambda\leq 0, 
 \end{cases}
 \]
 along with the matrix field $\A$, the vector field $a$ and the function $\eta$ defined in Proposition \ref{prop:main1}. 
\subsubsection*{\bf Case of $\lambda\in [0,1]$:}
The end point case, $\lambda=1$, is of course due to \cite{GIGAKOHN1985} and the below procedure recovers this classical result. Suppose $\Omega$ is a ball of radius $r>0$ centred at the origin.  With our choice of $\Phi$ the identity \eqref{id} reduces to  
\[
\int_{B_r} \la\nabla w, \A\nabla w\ra=\int_{\partial B_r}\eta
\]
for every $r>0$.  Moreover, by Proposition \ref{A+}$a)$, $\A(x)$ is positive definite for every $x\in \dot\R^n$. According to Proposition \ref{prop:main4} we may select a sequence of spheres $\partial B_{r_k}$ along which the boundary integral vanishes asymptotically. Thus, by the monotone convergence theorem, we may pass to the limit
\[
\int_{\R^n} \la\nabla w, \A\nabla w\ra=\lim_{k\to \infty}\int_{B_{r_k}} \la\nabla w, \A\nabla w\ra=0.
\]
This in turn forces $\nabla w\equiv 0$ everywhere, i.e. $w$ must be constant on $\R^n$.

\subsubsection*{\bf Case of $\lambda<  0$:}
When $\lambda$ passes zero the three constant solutions of \eqref{eq:eq} collapse to one null solution. As we will see in the proof of Proposition \ref{prop:main2} we find that $\nabla \cdot \Phi_0=n\rho$, which is strictly positive. The matrix field $\A$ and the boundary integral retain their properties but the quadratic term in \eqref{id} does not vanish in this case. Instead we have 
 \[
  \nabla \cdot a=\frac{\lambda}{2(p+1)}\nabla \cdot \Phi_0=-\frac{n|\lambda|}{2(p+1)}\rho,
 \]
 so that 
 \[
\int_{\R^n} \la\nabla w, \A\nabla w\ra+\frac{n|\lambda|}{p+1}\int_{\R^n} |w|^2\rho =0.
\]
Thus $w\equiv 0$ on $\R^n$ as required. 
\subsubsection*{\bf Case of $p=p_S$, $n\geq 4$ and $\lambda\in \left[\frac{3n}{2(n-1)},2\right]$:}We follow the same steps as in the case of $\lambda\in [0,1]$ ensuring that $\nabla\cdot a\equiv 0$ on $\R^n$ and that the boundary term vanishes asymptotically. By Proposition \ref{A+}$b)$, with these parameters, the matrix field $\A$ restricted to radial fields is negative definite outside the origin. Thus, for a radial solution $w$ we may still appeal to the monotone convergence theorem to conclude that it has to be constant. 
\end{proof}

\section{Proofs of Propositions \ref{prop:main1} through \ref{prop:main4}}
We will demonstrate arguments for Propositions \ref{prop:main1}, \ref{prop:main2}, \ref{A+}{\it a)} and \ref{prop:main4} first, followed by Proposition \ref{A+} {\it b)} and {\it c)} which require some preparation first. 

It will be convenient to rewrite \eqref{eq:eq} as
\begin{equation}\label{eq:main}
\Lop w+f_\lambda(w)=0 \quad \mbox{on}\quad \R^n,
\end{equation}
with 
\[
f_\lambda(s)=-\frac{\lambda}{p-1} s+|s|^{p-1}s. 
\]


\begin{proof}[Proof of Proposition \ref{prop:main1}]

Let $\Phi$ be a smooth vector field on $\R^n$. We will multiply \eqref{eq:main} through $\la \nabla w,\Phi\ra $ and integrate over $\Omega$. Since 
\[
\la \nabla w,  \la\nabla^2 w,\Phi\ra\ra=\frac{1}{2}\la \nabla |\nabla w|^2,\Phi\ra=\frac{1}{2}\nabla \cdot \bigg(|\nabla w|^2\Phi\bigg)-\frac{1}{2}|\nabla w|^2\nabla \cdot \Phi,
\]
by a repeated application of the divergence theorem we get
\begin{align*}
\int_{\Omega}\mathcal L w \langle \nabla w,\Phi\rangle&=\int_{\Omega}\nabla \cdot\left(\rho\nabla w\right)\left\langle \nabla w,\frac{\Phi}{\rho}\right\rangle=\int_{\partial \Omega}\langle \nabla w,\nu\rangle\langle \nabla w,\Phi\rangle-\int_{\Omega}\la\nabla w,\nabla \left\langle \nabla w,\frac{\Phi}{\rho}\right\rangle\ra \rho\\
&=\int_{\partial \Omega}\langle \nabla w,\nu\rangle\langle \nabla w,\Phi\rangle-\int_{\Omega}\la \nabla w,  \la\nabla^2 w,\Phi\ra\ra-\int_{\Omega}\la \nabla w, \la\rho\nabla \frac{\Phi}{\rho}, \nabla w\ra\ra\\
&=\int_{\partial \Omega}\langle \nabla w,\nu\rangle\langle \nabla w,\Phi\rangle-\frac{1}{2}|\nabla w|^2\langle \Phi,\nu\rangle+\int_{\Omega} \frac{1}{2} |\nabla w|^2\nabla\cdot\Phi-\la \nabla w, \la\rho\nabla \frac{\Phi}{\rho}, \nabla w\ra\ra
\end{align*}
and 
\[
\int_{\Omega}f_\lambda(w)\la \nabla w,\Phi\ra=\int_{\partial \Omega}
\Pi_\lambda(w)\la \Phi,\nu\ra-\int_{\Omega}\Pi_\lambda(w)\nabla \cdot \Phi,
\]
where $\Pi_\lambda$ is the primitive function of $f_\lambda$ given by
\[
\Pi_\lambda(s)=-\frac{\lambda}{2(p-1)}|s|^2+\frac{1}{p+1}|s|^{p+1}.
\]
Thus, we obtain our first integral identity 
\begin{align}\label{A}\tag{A}
\begin{split}
0=&\frac{1}{2}\int_{\Omega}  |\nabla w|^2\nabla\cdot\Phi-\int_{\Omega}\la \nabla w, \la\rho\nabla \frac{\Phi}{\rho}, \nabla w\ra\ra-\int_{\Omega}\Pi_\lambda(w)\nabla \cdot \Phi\\
&+\int_{\partial \Omega}\langle \nabla w,\nu\rangle\langle \nabla w,\Phi\rangle-\frac{1}{2}\int_{\partial \Omega}|\nabla w|^2\langle \Phi,\nu\rangle+\int_{\partial \Omega}
\Pi_\lambda(w)\la \Phi,\nu\ra.
\end{split}
\end{align}
Next, we will test \eqref{eq:main} with  $w\nabla \cdot\Phi$. Since 
\begin{align*}
\int_{\Omega}\mathcal L w \left(w\nabla \cdot \Phi\right)=&\int_{\partial \Omega}\la\nabla w,\nu\ra w\nabla \cdot\Phi-\int_{\Omega}|\nabla w|^2\nabla \cdot\Phi-\int_{\Omega}\la w\nabla w ,\rho\nabla\frac{\nabla \cdot\Phi}{\rho}\ra\\
=&\int_{\partial \Omega}\la\nabla w,\nu\ra w\nabla \cdot\Phi-\frac{1}{2}|w|^2\la \nabla\frac{\nabla \cdot\Phi}{\rho},\nu\ra\rho\\
&+\frac{1}{2}\int_{\Omega}|w|^2\nabla\cdot \left(\rho\nabla\frac{\nabla \cdot\Phi}{\rho}\right)-|\nabla w|^2\nabla \cdot\Phi
\end{align*}
we obtain the second integral identity 
\begin{align}\label{B}\tag{B}
\begin{split}
0=&-\int_{\Omega}|\nabla w|^2\nabla \cdot\Phi+\frac{1}{2}\int_{\Omega}|w|^2\nabla\cdot \left(\rho\nabla\frac{\nabla \cdot\Phi}{\rho}\right)+\int_{\Omega}f_\lambda(w)w\nabla \cdot\Phi\\
&+\int_{\partial \Omega}\la\nabla w,\nu\ra w\nabla \cdot\Phi-\frac{1}{2}\int_{\partial \Omega}|w|^2\la \nabla\frac{\nabla \cdot\Phi}{\rho},\nu\ra\rho. 
\end{split}
\end{align}
Note, that the following relation holds
\[
\frac{1}{p+1}f_\lambda(s) s-\Pi_\lambda(s)=\frac{\lambda}{2(p+1)}|s|^2
\]
and so we can combine \eqref{A} and \eqref{B} in such a way that integrals over $\Omega$ containing $|w|^{p+1}$ cancel out

\[
\eqref{A}+\frac{1}{p+1}\eqref{B}=0.
\]
This linear combination amounts to \eqref{id}.

\end{proof}

\begin{proof}[Proof of Proposition \ref{prop:main2}]
Set $Q_\mu=\frac{\nabla \cdot \Phi_\mu}{\rho}$, then the equation \eqref{phi} may be rewritten as 

\begin{equation}\label{eigen}
\Lop Q_\mu+\mu Q_\mu=0. 
\end{equation}
There is no need to approach this problem in any degree of generality as long as we can point to any solution that matches our requirements. The rotational symmetry of the problem strongly suggests that we may restrict our attention to \emph{radial fields}, i.e.\,fields of the form $\Phi_\mu(x):=x\sigma_\mu(r)$ for some smooth function $\sigma_\mu:[0,\infty)\mapsto \R$.  Then, we may also treat $Q_\mu$ as a function of the radius in which case  \eqref{eigen} reads
\begin{equation}\label{radeigen}
Q_\mu''+\left(\frac{n-1}{r}-\frac{r}{2}\right)Q_\mu'+\mu Q_\mu=0.
\end{equation}
Aiming at a smooth solution we equip the above with the requirement that  $Q_\mu'(0)=0$. We also set $Q_\mu(0)=n$. Other choices of the initial value are equivalent but this one makes calculations clearer and is immediately compatible with the fully worked out case $\mu=1$. Of course \eqref{eigen} brings to mind an eigenproblem for the Ornstein-Uhlenbeck operator and indeed a relevant theory may be developed in the setting of the weighted Lebesgue space $L^2_{\rho}(\R^n)$, see \cite{FILIPPASKOHN}. Analysis along these lines would force us to deal with a discrete spectrum instead of arbitrary $\mu$. However, there is no need to impose such integrability restrictions and without them solutions of \eqref{radeigen} are expressed in terms of the well known Kummer functions,
\begin{equation}\label{Q}
Q_\mu(r)=n M\left(-\mu,\frac{n}{2},\frac{r^2}{4}\right), 
\end{equation}
 i.e.\,confluent hypergeometric functions of the first kind. Their definition along with a small selection of their multiple special properties are provided in the Appendix. 
 
In order to have an exact form of the vector field $\Phi_\mu$ we need an expression for $\sigma_\mu$. We will show that 
\begin{equation}\label{sigma}
\sigma_\mu=M\left(1-\mu,\frac{n}{2}+1,\frac{r^2}{4}\right)\rho.
\end{equation}

Let us look again at the required condition $\nabla \cdot a_\mu\equiv 0$. If we postulate an alternative expression for the vector field, $\Phi_\mu=\rho\nabla \psi_\mu$ for some smooth radial function $\psi_\mu$, then condition \eqref{eigen} takes the form 
\[
\nabla \cdot\left(\rho\nabla\frac{\nabla \cdot\rho\nabla \psi_\mu}{\rho}+\mu\rho\nabla \psi_\mu\right)=\rho \Lop\left(\Lop \psi_\mu+\mu\psi_\mu\right)=0
\]
which is satisfied if we simply demand 
\[
Q_\mu+\mu\psi_\mu=\Lop \psi_\mu+\mu\psi_\mu \equiv const. 
\]
Taking cue from  the case $\mu=1$, where $\psi_1(r)=\frac{r^2}{2}$ we will choose the constant to be equal to the dimension and ask for $\psi_\mu(0)=0$. Then, a solution suitable for our purposes and consistent with $Q_\mu=n-\mu \psi_\mu$ is given by  
\[
\psi_\mu(r)=\frac{n}{\mu}-\frac{n}{\mu}M\left(-\mu,\frac{n}{2},\frac{r^2}{4}\right).
\]
Since $\psi_\mu$ is radial we can write 
\[
\rho\nabla \psi_\mu=\rho\psi_\mu'\frac{x}{r}=x\left(\frac{\psi_\mu'}{r}\rho\right)=:x\sigma_\mu . 
\]
Due to special identities satisfied by Kummer functions -- see \eqref{A.4} -- we arrive at
\[
\sigma_\mu=M\left(1-\mu,\frac{n}{2}+1,\frac{r^2}{4}\right)\rho,
\]
as required. As a side comment we note that since  
\begin{equation}\label{eq:div}
\nabla \cdot \Phi_\mu=n \sigma_\mu+r\sigma_\mu'=n M\left(-\mu,\frac{n}{2},\frac{r^2}{4}\right)\rho
\end{equation}
the following compact formula  follows 
\[
\dashint_{B_r}M\left(-\mu,\frac{n}{2},\frac{|y|^2}{4}\right)\rho\d y=M\left(1-\mu,\frac{n}{2}+1,\frac{r^2}{4}\right)\rho.
\]
Hence, we have established, that the vector field given by 
\[
\Phi_\mu(x)=x M\left(1-\mu,\frac{n}{2}+1,\frac{|x|^2}{4}\right)\rho
\]  
possesses the desired properties. 
\end{proof}

\begin{proof}[Proof of Proposition \ref{A+}a)]
By direct calculation we obtain 
 \begin{align*}
\rho\nabla \frac{\Phi_\mu}{\rho}&=\rho\nabla \left(x\frac{\sigma_\mu}{\rho}\right)=\sigma_\mu \Id+\rho \left(x\otimes\nabla \frac{\sigma_\mu}{\rho}\right)=\sigma_\mu \Id+r\rho \left(\frac{\sigma_\mu}{\rho}\right)'\nu\otimes \nu\\
&=\sigma_\mu \Id+r\left(\sigma_\mu'+\frac{r}{2}\sigma_\mu\right)\nu\otimes \nu,
\end{align*}
where $\nu=\frac{x}{r}$ is the exterior normal field to the sphere $\partial B_r$ and
\[
\nu\otimes \nu=\frac{1}{r^2}\begin{pmatrix}
x_1^2 & x_1 x_2 & x_1 x_3 & \cdots & x_1 x_n \\
x_2 x_1& x_2^2 & x_2 x_3 & \cdots & x_2 x_n \\
x_3 x_1& x_3 x_2& x_3^2 & \cdots & x_3 x_n \\
\vdots & \vdots & \vdots & \ddots & \vdots \\
x_n x_1 & x_n x_2& x_n x_3 & \cdots & x_n^2 \end{pmatrix}
\]
 is the outer product of $\nu$ with itself.  It follows after a rearrangement of terms and taking \eqref{eq:div} into account,  that for an arbitrary vector $\alpha\in \R^n$ we have 
\begin{equation}\label{eq:Asigma}
\left\langle \alpha,\A_\mu\alpha\right\rangle=|\alpha|^2 I_\mu+|\langle \alpha,\nu\rangle|^2J_\mu,
\end{equation}
with 
\begin{equation}\label{IJ}
I_\mu=\left(\frac{n}{p+1}-\frac{n-2}{2}\right) \sigma_\mu-\left(\frac{1}{2}-\frac{1}{p+1}\right)r\sigma_\mu' \quad\mbox{ and }\quad  J_\mu=r\left( \sigma_\mu'+\frac{r}{2}\sigma_\mu\right).
\end{equation}
We will now use properties of Kummer functions to argue that for $1<p\leq p_S$ and $\mu\in [0,1]$ functions $I_\mu$ and $J_\mu$ are strictly positive for all $r\geq 0$. Note, that $p\leq p_S$ simply means that the coefficient in front of $\sigma_\mu$ in $I_\mu$ is nonnegative. With the use of special identities \eqref{A.4} we get 
\begin{equation}\label{sigma'}
-r\sigma_\mu'=\frac{\frac{n}{2}+\mu}{\frac{n}{2}+1}M\left(\frac{n}{2}+\mu+1,\frac{n}{2}+2,-\frac{r^2}{4}\right)\frac{r^2}{2}
=\frac{\frac{n}{2}+\mu}{\frac{n}{2}+1}M\left(1-\mu,\frac{n}{2}+2,\frac{r^2}{4}\right)\frac{r^2}{2}\rho.
\end{equation}
and further
\begin{equation}\label{sigma'+}
\begin{split}
 \sigma_\mu'+\frac{r}{2}\sigma_\mu&=\frac{r}{2}\left[M\left(1-\mu,\frac{n}{2}+1,\frac{r^2}{4}\right)-\frac{\frac{n}{2}+\mu}{\frac{n}{2}+1}M\left(1-\mu,\frac{n}{2}+2,\frac{r^2}{4}\right)\right]\rho\\
&=\frac{r}{2}\frac{1-\mu}{\frac{n}{2}+1}M\left(2-\mu,\frac{n}{2}+2,\frac{r^2}{4}\right)\rho. 
\end{split}
\end{equation}
Since, $M\left(a,b,\frac{r^2}{4}\right)$ is strictly positive whenever $a,b\geq 0$ we conclude that $\mu\in \left[-\frac{n}{2},1\right]$ guarantees that $\sigma_\mu, -r\sigma_\mu'$ are strictly positive and $ \sigma_\mu'+\frac{r}{2}\sigma_\mu$ is nonnegative for all $r$. Hence, $\A_\mu$ is strictly positive on $\dot \R^n$.    

\end{proof}

\begin{proof}[Proof of Proposition \ref{prop:main4}]
First we simplify the expression \eqref{eta} and rephrase it in terms of spherical means, which will be more convenient to work with. Recall \eqref{Q} and \eqref{sigma}, i.e.\,the formulae for $Q_\lambda$ and $\sigma_\lambda$. It follows that 
 \[
 \rho Q_\lambda'=- \lambda r M \left(1-\lambda,\frac{n}{2}+1,\frac{r^2}{4}\right)\rho=-\lambda r \sigma_\lambda
 \]
 and in effect  
 \begin{align*}
 \langle a_\lambda,\nu\rangle&=\frac{1}{2(p+1)}\la \rho Q_\lambda'+\lambda \sigma_\lambda x,\nu\ra=\frac{1}{2(p+1)}\left(\rho Q_\lambda' +\lambda r \sigma_\lambda \right)=0.
 \end{align*}
Hence, we may write 

\begin{align*}
 \frac{1}{|\partial B_1|}\int_{\partial B_r}\eta_\lambda=&  r^n\sigma_\lambda\dashint_{\partial B_r}|\langle \nabla w,\nu\rangle |^2-\frac{1}{2}|\nabla w|^2+\frac{1}{p+1}f_\lambda(w)w+\frac{r^{n-1}Q_\lambda \rho}{p+1} \dashint_{\partial B_r}\la\nabla w,\nu\ra w.
 \end{align*}
 By the choice of $\Phi_\lambda$ and through \eqref{id} we know that  $\int_{\partial B_r}\eta_\lambda\geq 0$. In order to get an upper bound we first apply  a straightforward estimate
\[
 \frac{1}{|\partial B_1|}\int_{\partial B_r}\eta_\lambda\leq   \frac{1}{2}r^n|\sigma_\lambda|\dashint_{\partial B_r}|\nabla w|^2+\frac{r^n|\sigma_\lambda|}{p+1}\dashint_{\partial B_r}|f_\lambda(w)w|+\frac{r^{n-1}|Q_\lambda| \rho}{2(p+1)} \dashint_{\partial B_r}|\nabla w|^2+|w|^2.
 \]
 Next, observe that due to \eqref{A.5} we have the following asymptotic behaviour 
 \begin{align*}
 r^n\sigma_\lambda&=\begin{cases}\mathcal{O}\left(r^{-2\lambda}\right)&\mbox{for }\lambda\neq 1,2,\dots,\\
 \mathcal{O}\left(r^{n-2\lambda+2}\rho\right)&\mbox{otherwise}
 \end{cases}\nonumber \\
 \mbox{ and }\label{asym}\\
  r^{n-1}Q_\lambda \rho&=\begin{cases}\mathcal{O}\left(r^{-2\lambda-1}\right)&\mbox{for }\lambda\neq 0,1,\dots,\\
 \mathcal{O}\left(r^{-1-2\lambda}\rho\right)&\mbox{otherwise}.\end{cases}\nonumber
 \end{align*}
 This means that 
 \[
\int_{\partial B_r}\eta_\lambda\leq  \mathcal{O}\left(r^{-2\lambda}\right)\dashint_{\partial B_r}|\nabla w|^2+|w|^2+|w|^{p+1}.
 \]
The bound is significantly better when $\lambda$ happens to be a positive integer but for our purposes it suffices to use this single estimate. At this point we split the argument into two cases. First we will deal with a more straightforward case of $\lambda>0$. 

 Our solution is uniformly bounded and if we could moreover assert that spherical means of $|\nabla w|^2$ remain uniformly bounded for some sequence of radii $r_k\to\infty$, then  on such a sequence we would find 
\[
 \int_{\partial B_{r_k}}\eta_\lambda\to 0\quad \mbox{ as }\quad k\to \infty.
\]
Now, by assumption, we may choose a finite $M>0$ such that
\[
\dashint_{\partial B_r}|w|^{p+1} \leq M^{p+1} 
\]
for all $r\geq 0$. Define a smooth function $h(r)=\frac{1}{2}\dashint_{\partial B_r}|w|^{2}$. Due to the Jensen inequality we have
\[
h(r)=\frac{1}{2}\left[\left(\dashint_{\partial B_r}|w|^{2}\right)^\frac{p+1}{2}\right]^\frac{2}{p+1}\leq  \frac{1}{2}\left(\dashint_{\partial B_r}|w|^{p+1}\right)^\frac{2}{p+1}\leq \frac{1}{2}M^2
\]
uniformly in $r$. Moreover
\[
h'(r)=\dashint_{\partial B_r}\langle \nabla w,\nu\rangle w,
\]
so that, by the mean value theorem, we may find a sequence of radii $r_k\in (k,k+1)$, $k=0,1,2,\dots$, such that 
\[
h'(r_k)=h(k+1)-h(k)\quad \mbox{and}\quad |h'(r_k)|\leq M^2.
\] 
Next, multiply \eqref{eq:main} by $w$ and integrate over $B_r$ to obtain
\[
\int_{B_r} |\nabla w|^2=\left(\frac{n}{4}-\frac{\lambda}{p-1}\right)\int_{B_r}|w|^2+\int_{B_r}|w|^{p+1}-\frac{r}{4}\int_{\partial B_r}|w|^2+\int_{\partial B_r}\langle \nabla w,\nu\rangle w.
\]
Since $n|B_r|=r|\partial B_r|$, there follows
\[
\dashint_{B_r} |\nabla w|^2=\left(\frac{n}{4}-\frac{\lambda}{p-1}\right)\dashint_{B_r}|w|^2+\dashint_{B_r}|w|^{p+1}-\frac{n}{4 }\dashint_{\partial B_r}|w|^2+\frac{n}{r}h'(r)
\]
and further due to 
\[
\dashint_{B_r}|w|^{p+1}=\frac{1}{|B_r|}\int_0^r \left(|\partial B_s|\dashint_{\partial B_s}|w|^{p+1}\right)  \leq M^{p+1}
\]
\Big(and analogously for $\dashint_{B_r}|w|^{2}$\Big) we find 
\[
\dashint_{B_{r}} |\nabla w|^2\leq \left(\frac{n}{4}-\frac{\lambda}{p-1}\right)M^2+M^{p+1}+\frac{n}{r}|h'(r)|.
\]
If we now consider a sequence of balls $B_{r_k}$ and take into account that $\lambda$ is positive, then we get the bound
\[
\dashint_{B_{r_k}} |\nabla w|^2\leq \left(\frac{n}{r_k}+\frac{n}{4}\right) M^2+M^{p+1}\leq N
\]
for some finite $N$ independent of $k$ or $\lambda$. Since the increments between successive radii are uniformly bounded it also follows that for  $r\in (r_k,r_{k+1})$
\begin{align*}
\dashint_{B_{r}} |\nabla w|^2=&\frac{|B_{r_{k+1}}|}{|B_{r}|}\left(\frac{1}{|B_{r_{k+1}}|}\int_{B_{r}} |\nabla w|^2\right)\leq \left(\frac{r_{k+1}}{r}\right)^n\dashint_{B_{r_{k+1}}} |\nabla w|^2\\
&\leq \left(\frac{r_{k}+2}{r_k}\right)^n\dashint_{B_{r_{k+1}}} |\nabla w|^2\leq \left(\frac{r_{0}+2}{r_0}\right)^n N=:N'.
\end{align*}
Since $N'$ does not depend on $k$ we get the uniform estimate $\dashint_{B_r} |\nabla w|^2\leq N'$. 

This implies however that there exists a further (relabelled) sequence $r_k\to \infty$ such that $\dashint_{\partial B_{r_k}}|\nabla w|^2\leq 2 N'$ uniformly in $k$. For suppose there did not exist such a sequence, i.e.\,there was a radius $R$ such that $\dashint_{\partial B_{r}}|\nabla w|^2>2N'$ for $r>R$, then 
\begin{align*}
N' &\geq \dashint_{B_r} |\nabla w|^2=\frac{1}{|B_r|}\int_0^r \left(|\partial B_s|\dashint_{\partial B_s} |\nabla w|^2\right)=\frac{n}{r^n}\int_0^r s^{n-1}\dashint_{\partial B_s} |\nabla w|^2\\
&\geq \frac{n}{r^n}\int_{R}^r s^{n-1}\dashint_{\partial B_s} |\nabla w|^2>2N'\left(1-\frac{R^n}{r^n}\right) 
\end{align*}
which is false for sufficiently large $r$. Now let $r_k\to \infty$ be the sequence implied in the above contradiction. The above argument yields 
 \begin{equation}\label{etabound}
\int_{\partial B_{r_k}}\eta_\lambda\leq  \mathcal{O}\left(r_k^{-2\lambda}\right)
 \end{equation}
which is enough as long as $\lambda>0$ but requires a refinement when $\lambda$ vanishes.

Set $\lambda=0$, multiply equation \eqref{eq:main} through $w \sigma_0$ and integrate over $B_{r_k}$. Then 
\begin{align*}
 \frac{1}{2}\int_{B_{r_k}}\nabla \cdot\left(\rho\nabla\frac{\sigma_0}{\rho}\right)|w|^2+\int_{B_{r_k}}|w|^{p+1}\sigma_0=&\int_{B_{r_k}}|\nabla w|^2\sigma_0-\int_{\partial B_{r_k}}\langle \nabla w,\nu\rangle w\sigma_0\\
 &-\frac{1}{2}\int_{\partial B_{r_k}}\left\langle\rho\nabla\frac{\sigma_0}{\rho},\nu\right\rangle|w|^2.
\end{align*}
First, let us estimate the boundary terms. To begin with 
\[
\left|\int_{\partial B_{r_k}}\langle \nabla w,\nu\rangle w\sigma_0\right|\leq \frac{1}{2}\sigma_0(r_k)\int_{\partial B_{r_k}}|\nabla w|^2+|w|^2=\mathcal{O}\left(r_k^{-1}\right)\dashint_{\partial B_{r_k}}|\nabla w|^2+|w|^2.
\]
Next, note that by  \eqref{sigma'+}
\[
\la\rho\nabla\frac{\sigma_0}{\rho},\nu\ra=\la\nabla \sigma_0+\frac{x}{2}\sigma_0,\nu\ra=\sigma_0'+\frac{r}{2}\sigma_0=\frac{r}{n+2}M\left(2,\frac{n}{2}+2,\frac{r^2}{4}\right)\rho=\mathcal{O}\left(r^{1-n}\right).
\]
Hence, 
\begin{align*}
\left|\int_{\partial B_{r_k}}\left\langle\rho\nabla\frac{\sigma_0}{\rho},\nu\right\rangle|w|^2\right|&\leq \mathcal{O}\left(r_k^{1-n}\right)\int_{\partial B_{r_k}}|w|^2=\mathcal{O}\left(1\right)\dashint_{\partial B_{r_k}}|w|^2.
\end{align*}

Observe that by \eqref{sigma} we have $\frac{\sigma_0}{\rho}=M\left(1,\frac{n}{2}+1,\frac{r^2}{4}\right)$. On the other hand this Kummer function solves the equation  

\[
\nabla\cdot\left(\rho\nabla \frac{\sigma_0}{\rho}\right)=-\frac{2 }{|x|^2}\langle \nabla \sigma_0,x\rangle+\sigma_0=-\frac{2 }{r}\sigma_0'+\sigma_0,
\]
see \eqref{A.3} in the Appendix. Since $\sigma_0'\leq 0$ we have 
\[
 \frac{1}{2}\int_{B_{r_k}}\nabla \cdot\left(\rho\nabla\frac{\sigma_0}{\rho}\right)|w|^2+\int_{B_{r_k}}|w|^{p+1}\sigma_0\geq \frac{1}{2}\int_{B_{r_k}}\left(|w|^2+|w|^{p+1}\right)\sigma_0.
\]
This combined with our estimates of the boundary terms yield
\[
\frac{1}{2}\int_{B_{r_k}}\Big(|w|^2+|w|^{p+1}\Big)\sigma_0\leq \int_{B_{r_k}}|\nabla w|^2\sigma_0+\mathcal{O}\left(r_k^{-1}\right)\dashint_{\partial B_{r_k}}\Big(|\nabla w|^2+|w|^2\Big)+\mathcal{O}\left(1\right)\dashint_{\partial B_{r_k}}|w|^2,
\]
which may be further moulded into 
\begin{equation}\label{mould}
\int_{B_{r_k}}\left(|\nabla w|^2+|w|^2+|w|^{p+1}\right)\sigma_0\leq 3\int_{B_{r_k}}|\nabla w|^2\sigma_0+\mathcal{O}\left(1\right)\dashint_{\partial B_{r_k}}|\nabla w|^2+|w|^2.
\end{equation}
Through our choice of the vector field $\Phi_0$, \eqref{id}, \eqref{eq:Asigma} and since by \eqref{etabound} we have $\int_{\partial B_{r_k}}\eta_0=\mathcal{O}(1)$ we find
\[
\int_{B_{r_k}}\langle \nabla w,\A_0 \nabla w\rangle=\int_{B_{r_k}}|\nabla w|^2 I_0+|\langle \nabla w,\nu\rangle|^2J_0=\mathcal{O}(1)
\]
uniformly in $k$. By Proposition \ref{A+}\emph{a)} the matrix field $\A_0$ is positive definite outside the origin and so we may use the monotone convergence theorem and \eqref{IJ} to conclude that 
\begin{equation}\label{yeah}
\int_{\R^n}|\nabla w|^2 I_0=\left(\frac{n}{p+1}-\frac{n-2}{2}\right)\int_{\R^n}|\nabla w|^2 \sigma_0-\left(\frac{1}{2}-\frac{1}{p+1}\right)\int_{\R^n}|\nabla w|^2r\sigma_0'<\infty.
\end{equation}
We have  established in  \eqref{sigma'} in the proof of Proposition \ref{prop:main2}) that $\sigma_0'\leq 0$ and so we readily see that $\int_{\R^n}|\nabla w|^2 \sigma_0$ is finite at least when $p<p_S$. However, the same is true in the critical case as the following argument shows. Recall that from \eqref{eq:div} there follows 
\[
\sigma_0=-\frac{1}{n}r\sigma_0' +M\left(0,\frac{n}{2},\frac{r^2}{4}\right)\rho=-\frac{1}{n}r\sigma_0'+\rho.
\]
Also, from \eqref{sigma'+} we get
\[
-\frac{1}{n}r\sigma_0'=\frac{1}{n+2}M\left(1,\frac{n}{2}+2,\frac{r^2}{4}\right)\frac{r^2}{2}\rho
\]
and it is clear from \eqref{A.1} that for $r>1$
\[
\rho(r)\leq r^2 M\left(1,\frac{n}{2}+2,\frac{r^2}{4}\right)=-\frac{2(n+2)}{n}r\sigma_0'.
\]
Thus, $\sigma_0\leq -\frac{2n+3}{n}r\sigma_0'$ outside the unit ball. Without loss of generality we may assume that $r_0>1$.  Then,

\begin{align*}
\int_{B_{r_k}}|\nabla w|^2\sigma_0&=\int_{B_{1}}|\nabla w|^2\sigma_0+\int_{B_{r_k}\setminus B_{1}}|\nabla w|^2\sigma_0\leq C_1-C_2\int_{B_{r_k}\setminus B_{1}}|\nabla w|^2r\sigma_0'
\end{align*} 
for some positive constants $C_1$ and $C_2$ independent of $k$. Due to \eqref{yeah} we have
\[
0\leq -\int_{B_{r_k}\setminus B_{1}}|\nabla w|^2r\sigma_0'\leq -\int_{\R^n}|\nabla w|^2r\sigma_0'<\infty. 
\]
Finally, coming back to \eqref{mould}, we see that we can pass to the limit with $r_k$ to get the estimate 
\[
\int_{\R^n}\Big(|\nabla w|^2+|w|^2+|w|^{p+1}\Big)\sigma_0<\infty.
\]
Since $\sigma_0\in \mathcal{O}\left(r^{-n}\right)$, this in turn implies that 
\begin{align*}
\int_{1}^\infty s^{n-1}\mathcal{O}\left(s^{-n}\right)\dashint_{\partial B_s}|\nabla w|^2+|w|^2+|w|^{p+1}= \int_{\R^n\setminus{B_1}}\Big(|\nabla w|^2+|w|^2+|w|^{p+1}\Big)\sigma_0<\infty.
\end{align*}
But, the left hand side cannot be finite unless there exists another (relabelled) sequence $r_k$ such that 
 
\[
\dashint_{\partial B_{r_k}}|\nabla w|^2+|w|^2+|w|^{p+1}\rightarrow 0 
\]
as $r_k\to \infty$. This concludes the proof.

\end{proof}

We will now turn our attention to the regime whereby $p$ is Sobolev critical and $\lambda>1$. For the remainder of the demonstration it is assumed that $w$ is radial. 

In order to simplify the notation we will from now on abbreviate $M\left(a,b,\frac{r^2}{4}\right)$ to $M(a,b)$. We further denote
\begin{equation}\label{us}
u_1=M\left(2-\lambda,\frac{n}{2}+2\right), \quad u_2=M\left(1-\lambda,\frac{n}{2}+2\right) \quad \mbox{and} \quad u_3=M\left(1-\lambda,\frac{n}{2}+1\right).
\end{equation}

\begin{proof}[Proof of Proposition \ref{A+}b)]

First note that when $w$ is radial then our expression \eqref{eq:Asigma} for the bilinear form associated with the field $\A_\lambda$ simplifies to give
\begin{equation}\label{st}
\int_{B_r}\left\langle\nabla w,\A_\lambda\nabla w\right\rangle  =\int_{B_r}|\nabla w|^2 I_\lambda+|\langle \nabla w,\nu\rangle |^2 J_\lambda= \int_{B_r}|w'|^2\left( I_\lambda+J_\lambda\right).
\end{equation}
Set 
\begin{equation}\label{flambda}
\Pi_\lambda=I_\lambda+J_\lambda=\frac{\frac{n}{2}+\lambda}{\frac{n}{2}+1}\left(\frac{1-\lambda}{\frac{n}{2}+\lambda}u_1+\frac{1}{n}u_2\right)\frac{r^2}{2}\rho+\left(\frac{n}{p+1}-\frac{n-2}{2}\right)u_3\rho. 
\end{equation}
Since $p=p_S$ and $n\geq 3$ the expression for $\Pi_\lambda$ further simplifies to  
\[
\Pi_\lambda=\frac{\frac{n}{2}+\lambda}{\frac{n}{2}+1}\left(\frac{1-\lambda}{\frac{n}{2}+\lambda}u_1+\frac{1}{n}u_2\right)\frac{r^2}{2}\rho.
\]
We recognise now that 
\[
\frac{1-\lambda}{\frac{n}{2}+\lambda}u_1(0)+\frac{1}{n}u_2(0)=\frac{1-\lambda}{\frac{n}{2}+\lambda}+\frac{1}{n}\leq 0\quad \mbox{whenever}\quad  \lambda\geq \frac{3n}{2(n-1)}=:\lambda_*.
\]
Also, 
\[
\left(\frac{1-\lambda}{\frac{n}{2}+\lambda}u_1+\frac{1}{n}u_2\right)'=-\frac{\lambda-1}{\frac{n}{2}+2}\left[\frac{2-\lambda}{\frac{n}{2}+\lambda}M\left(3-\lambda,\frac{n}{2}+3\right)+\frac{1}{n}M\left(2-\lambda,\frac{n}{2}+3\right)\right]\frac{r}{2}.
\]
Both Kummer functions present on the right-hand-side are strictly positive for $\lambda\leq 2$ and the whole expression is strictly negative for $\lambda\in (1,2]$ and $r>0$. Thus, whenever $\lambda\in \left[\lambda_*,2\right]$ we have that $\frac{1-\lambda}{\frac{n}{2}+\lambda}u_1+\frac{1}{n}u_2$ starts off nonpositive and then decreases as $r$ grows.  This interval is nonempty for $n\geq 4$. We find therefore that in these circumstances $\Pi_\lambda(r)$ is of constant negative sign for $r>0$.

\end{proof}

In the Sobolev critical case with $\lambda>1$ the methodology put forward in this paper is sensitive enough to rule out nontrivial radial solutions for a range of parameters specified in the Liouville Theorem. Now we will show that outside of this range the technique is \emph{inconclusive}, by which we mean that, if there is a Liouville theorem for parameters falling outside of the indicated range, then its demonstration \emph{cannot} rely on the definiteness of the $\A$ matrix field. In order to show it, i.e.\,prove Proposition \ref{A+}{\it c)}, we first need to gather some observations and tools. 

\subsection{Behaviour of $\Pi_\lambda$ at the origin}
We know already what happens near the origin for dimensions $n\geq 3$:
\[
r^{-2}\Pi_\lambda(r)\xrightarrow[r\to 0]{}\frac{\frac{n}{2}+\lambda}{n+2} \left(\frac{1-\lambda}{\frac{n}{2}+\lambda}+\frac{1}{n}\right),
\]
so that 
\[
\begin{cases} \Pi_\lambda(0)\geq 0&\quad \mbox{for}\quad \lambda\leq \lambda_*\\
\Pi_\lambda(0)< 0&\quad \mbox{for}\quad \lambda> \lambda_*.
\end{cases}
\]
When $n=1$ or $2$ then the situation is simpler: 
\[
\Pi_\lambda(r)\xrightarrow[r\to 0]{} 
\frac{n}{p+1}-\frac{n-2}{2}>0.
\]

\subsection{Asymptotic behaviour of  $\A$}

\begin{lem}\label{asymptotic}
Let $n\geq 1$ and $\lambda>1$, then as $r\to \infty$ we have
\begin{equation}\label{ranges}
\mbox{asymptotic sign of } \Pi_\lambda=\begin{cases}&+\mbox{ for }\lambda\in (2,3]\cup (4,5]\cup (6,7]\cup\dots, \\
&-\mbox{ for }\lambda\in (1,2]\cup (3,4]\cup (5,6]\cup\dots.
\end{cases}
\end{equation}
\end{lem}
\begin{proof}
Assume first that $a\neq 0,-1,-2,\dots$, then by \eqref{A.5} we have 
\[
M\left(a,b\right)\rho\left(\frac{r^2}{4}\right)^{b-a}\longrightarrow \quad \frac{\Gamma(b)}{\Gamma(a)}
\]
as $r\to \infty$ and this implies for \eqref{us} the following asymptotics. 
\begin{equation}\label{3}
\begin{cases}
u_1\rho\left(\frac{r^2}{4}\right)^{\frac{n}{2}+\lambda}&\longrightarrow \quad \frac{\Gamma\left(\frac{n}{2}+2\right)}{\Gamma\left(2-\lambda\right)},\\
u_2\rho\left(\frac{r^2}{4}\right)^{\frac{n}{2}+\lambda+1}&\longrightarrow \quad \frac{\Gamma\left(\frac{n}{2}+2\right)}{\Gamma\left(1-\lambda\right)},  \\
u_3\rho\left(\frac{r^2}{4}\right)^{\frac{n}{2}+\lambda}&\longrightarrow \quad \frac{\Gamma\left(\frac{n}{2}+1\right)}{\Gamma\left(1-\lambda\right)}.
\end{cases}
\end{equation}

 Our expression \eqref{flambda} may be rephrased as

\begin{align*}
\left(\frac{r^2}{4}\right)^{\frac{n}{2}+\lambda-1}\Pi_\lambda(r)=&2\left(\frac{\frac{n}{2}+\lambda}{\frac{n}{2}+1}\right)\left(\frac{1-\lambda}{\frac{n}{2}+\lambda}\right) u_1\rho \left(\frac{r^2}{4}\right)^{\frac{n}{2}+\lambda}+\frac{2}{n}\left(\frac{\frac{n}{2}+\lambda}{\frac{n}{2}+1}\right)u_2\rho \left(\frac{r^2}{4}\right)^{\frac{n}{2}+\lambda}\\
&+\left(\frac{n}{p+1}-\frac{n-2}{2}\right)u_3\rho\left(\frac{r^2}{4}\right)^{\frac{n}{2}+\lambda-1}
\end{align*}
and due to \eqref{3} we find that
\[
\mbox{asymptotic sign of } \Pi_\lambda=\mbox{sign of }\frac{1-\lambda}{\Gamma(2-\lambda)}=\begin{cases}&+\mbox{ for }\lambda\in (2,3)\cup (4,5)\cup (6,7)\cup\dots, \\
&-\mbox{ for }\lambda\in (1,2)\cup (3,4)\cup (5,6)\cup\dots.
\end{cases} 
\]
When $\lambda=m=2,3,4,\dots$, then by \eqref{A.1} the Kummer functions involved reduce to polynomials: 
\[
\begin{cases}
u_1\left(\frac{r^2}{4}\right)^{2-m}&\longrightarrow \quad (-1)^{m-2} \frac{(m-2)!}{\left(\frac{n}{2}+2\right)_{m-2}},\\
u_2\left(\frac{r^2}{4}\right)^{1-m}&\longrightarrow \quad (-1)^{m-1}\frac{(m-1)!}{\left(\frac{n}{2}+2\right)_{m-1}},  \\
u_3\left(\frac{r^2}{4}\right)^{1-m}&\longrightarrow \quad (-1)^{m-1}\frac{(m-1)!}{\left(\frac{n}{2}+1\right)_{m-1}}.
\end{cases}
\]
Therefore
\begin{align*}
\frac{1}{\rho}\left(\frac{r^2}{4}\right)^{-m}\Pi_\lambda(r)=&2\left(\frac{\frac{n}{2}+m}{\frac{n}{2}+1}\right)\left(\frac{1-m}{\frac{n}{2}+m}\right) u_1 \left(\frac{r^2}{4}\right)^{1-m}+\frac{2}{n}\left(\frac{\frac{n}{2}+m}{\frac{n}{2}+1}\right)u_2 \left(\frac{r^2}{4}\right)^{1-m}\\
&+\left(\frac{n}{p+1}-\frac{n-2}{2}\right)u_3\left(\frac{r^2}{4}\right)^{-m}.
\end{align*}
Thus 
\[
\mbox{asymptotic sign of } F_m=(-1)^{m-1}=\begin{cases}&+\mbox{ when }m \mbox{ is odd}, \\
&-\mbox{ when }m\mbox{ is even}.
\end{cases}
\]
\end{proof}

\subsection{Sturmian comparison}

One of the key ingredients in the below considerations will be the relative position of roots of specific Kummer functions that make up the expression for $\Pi_\lambda$. To this end we will take advantage of yet another identity, the Picone identity \cite{PICONE}, which states that whenever $v,\omega, K v'$ and $L \omega'$ are differentiable functions of $r$ and $\omega(r)\neq 0$, then 
\[
v\left( K v'\right)'-\frac{v^2}{\omega}\left( L \omega'\right)'+(K-L)v'^2+L\left(v'-\omega'\frac{v}{\omega}\right)^2=\left[\frac{v}{\omega}(K v'\omega-L v \omega')\right]'.
\]

\begin{lem}\label{picone}
Let $\lambda>2$, then there exists a point where $\Pi_\lambda$ attains a negative value and if moreover $n\geq 3$ and $\lambda>3$ then $\Pi_\lambda$ attains positive values as well.  
\end{lem}

\begin{proof}
Due to \eqref{A.2} we know that our three Kummer functions satisfy:
\begin{equation}\label{u123}
\begin{cases}
\left(r^{n+3}\rho u_1'\right)'+(\lambda-2)\rho r^{n+3}u_1=0,&\\
\left(r^{n+3}\rho u_2'\right)'+(\lambda-1)\rho r^{n+3}u_2=0,&\\
\left(r^{n+1}\rho u_3'\right)'+(\lambda-1)\rho r^{n+1}u_3=0,&\\
\end{cases}
\end{equation}
with $u_i(0)=1$ and $u_i'(0)=0$ for $i=1,2,3$. \\
Assume that $\lambda>0$ and the dimension may be taken arbitrary. Due to \eqref{A.6} we know that $u_1$ has at least one positive root. Let $\kappa$ denote the smallest of them. We will show that $u_2$ has a root in the interval $(0,\kappa)$. Suppose not, then $u_2>0$ in this interval. In such a case we may employ the Picone identity with 
\[
v=u_1, \quad \omega=u_2 \quad {and}\quad K=L=r^{n+3}\rho
\]
to get
\begin{equation}\label{pic}
u_1^2\rho+r^{n+3}\rho\left(u_1'-u_2'\frac{u_1}{u_2}\right)^2=\left[r^{n+3}\rho\frac{u_1}{u_2}( u_1'u_2-u_1 u_2')\right]',
\end{equation}
which, when integrated between the origin and $\kappa$, yields
\[
\int_0^{\kappa}u_1^2\rho+\int_0^{\kappa}r^{n+3}\rho\left(u_1'-u_2'\frac{u_1}{u_2}\right)^2=0.
\]
This integral identity implies that $u_1\equiv 0$ in $[0,\kappa]$ - a contradiction. Hence, there must have been a root of $u_2$ in $(0,\kappa)$. It is of no relevance to this argument whether $u_2$ changes sign more than once in this interval. Let us however denote its first (possibly only) root by $\iota$. Since $u_1$ is strictly positive in $(0,\kappa)$,  at this root in particular we find that
\[
\frac{1-\lambda}{\frac{n}{2}+\lambda}u_1(\iota)+\frac{1}{n}u_2(\iota)<0. 
\]

Next we will argue that $u_3$ has a unique root in the interval $[0,\iota]$.   
First, observe that the equation for $u_2$ may  be rewritten as

\[
\left(r^{n+1}\rho u_2'\right)'+2r^{n}\rho u_2'+(\lambda-1)\rho r^{n+1}u_2=0.
\]
Suppose that $u_3$ remains strictly positive in $(0,\iota)$. Then, if we set 
\[
v=u_2, \quad \omega=u_3 \quad {and}\quad K=L=r^{n+1}\rho
\]
and resort to the Picone identity again, we get 
\[
-2r^n\rho u_2'u_2+r^{n+3}\rho\left(u_2'-u_3'\frac{u_2}{u_3}\right)^2=\left[\frac{u_2}{u_3}r^{n+1}\rho( u_2'u_3- u_2 u_3')\right]'.
\]
Now integrate the above over $(0,\iota)$ to find
\begin{equation}\label{blah}
-2\int_{0}^\iota r^n\rho u_2'u_2+\int_{0}^\iota r^{n+1}\rho\left(u_2'-u_3'\frac{u_2}{u_3}\right)^2=0.
\end{equation}
We know that $u_2>0$ in $[0,\iota)$ and moreover from \eqref{u123} we learn that 
\[
u_2'(r)=\frac{1-\lambda}{r^{n+3}\rho(r)}\int_0^r\rho(s) s^{n+3}u_2(s)\d s,
\]
which yields in particular that $u_2'<0$ in $(0,\iota]$. Thus, \eqref{blah} cannot hold and in effect $u_3$ must have had a root in $(0,\iota)$. Note, that the differentiation rule for Kummer functions tie $u_1$ and $u_3$ in the relation
\[
\frac{r}{2}u_1(r)=\frac{\frac{n}{2}+1}{1-\lambda}u_3'(r).
\]
Since $u_1$ is strictly positive in $[0,\iota]$ we conclude that $u_3$ is strictly decreasing in $(0,\iota]$. Hence, $u_3$ may have only one root in this interval. In particular $u_3(\iota)<0$. 

Consider now the value of $\Pi_\lambda$ at $r=\iota$. We have 
\[
\frac{1}{\rho(\iota)}\Pi_\lambda(\iota)=\begin{cases}
&-\frac{\lambda-1}{3}u_1(\iota)\iota^2+\left(\frac{1}{p+1}+\frac{1}{2}\right)u_3(\iota)\quad \mbox{ for }\quad n=1,\\
&-\frac{\lambda-1}{4}u_1(\iota)\iota^2+\frac{2}{p+1} u_3(\iota)\quad \mbox{ for }\quad n=2,\\
&-\frac{\lambda-1}{n+2}u_1(\iota)\iota^2\quad \mbox{ otherwise}.
\end{cases}
\]
Since $u_1(\iota)>0$ and $u_3(\iota)<0$ we find $\Pi_\lambda(\iota)<0$ in every dimension.  

Next, assume that $n\geq 3$ and $\lambda>3$. Observe that $u_1$ has at least two roots. Select the first two $0<\kappa<\kappa'$. We already know that the first root of $u_2$ satisfies $\iota \in (0,\kappa)$. The reasoning based on the Picone identity shows that there needs to be another root of $u_2$ - call it $\iota'$ (not necessarily the consecutive one) - that lies in $(\kappa,\kappa')$, because otherwise we could integrate \eqref{pic} in this interval and arrive at a contradiction. Since $u_1$ is negative in  $(\kappa,\kappa')$ we find that $\Pi_\lambda(\iota')>0$.  

\end{proof}

We are in position to finish the demonstration of Proposition \ref{A+}. 

\begin{proof}[Proof of Proposition \ref{A+}c)]
Clearly for $n\geq 3$ and $\lambda>3$ Lemma \ref{picone} already proves the result. In other cases the argument will differ in details.  
\subsubsection*{\bf Case of $n=1$ or $2$:} First observe that no matter what values $\lambda$ or $p$  may assume, $\Pi_\lambda$ is strictly positive  at the origin. However, when $\lambda\in (1,2]$ then $\Pi_\lambda$ attains negative values sufficiently far from the origin by \eqref{ranges}. Then, for $\lambda>2$, we may invoke Lemma \ref{picone} and again assert that $\Pi_\lambda$ must have changed sign at some point. 
 
\subsubsection*{\bf Case of $n=3$:} In this dimension $\lambda_*=\frac{9}{4}$ and we see in particular that $\Pi_\lambda(0)>0$ for $\lambda\in (1,2]$. On the other hand we already know that for such $\lambda$ and $r$ large enough $\Pi_\lambda$ is negative. 
If $\lambda\in \left(2,\frac{9}{4}\right)$ then $\Pi_\lambda(0)$ remains positive. However, by Lemma \ref{picone} $\Pi_\lambda(r)$ must change sign to negative somewhere. Further still when $\lambda\in \left[\frac{9}{4},3\right]$ there is a disparity between the just mentioned negative values and the positive value expected asymptotically by Lemma \ref{asymptotic}. 
\subsubsection*{\bf Case of $n\geq 4$ and $\lambda\notin [\lambda_*,2]$:}
Suppose first that $\lambda<\lambda_*$. At the origin we find that $\Pi_\lambda$ is positive but we already know that $\Pi_\lambda$ will eventually become negative. For $\lambda\in (2,3]$ the situation swaps with negative values at the origin and positive asymptotically. 

Hence, outside of the conditions specified in $a)$ and $b)$ of Proposition \ref{prop:main2} the vector field $\Phi_\lambda$ gives rise to a matrix field $\A_\lambda(r)$ that when applied to radial fields may be positive or negative definite depending on the distance from the origin.

\end{proof}

\section{Discussion}
A question arises whether confluent hypergeometric functions are a natural choice of building blocks for our integral identities. An application of Kummer functions in the context of the Rellich-Pohozaev technique is to the best of our knowledge new. Nevertheless, it should be mentioned, that these special functions do appear in the study of singularities for the Fujita equation for example in the improved blow-up asymptotics in the Sobolev subcritical regime \cite{FILIPPASKOHN}, construction of type II blow-ups in the Joseph-Lundgren supercritical regime \cite{HV} and in the study of self-similar blow-ups in the Lepin supercritical regime \cite{MIZOGUCHI2009,PLECHACSVERAK2003}. In all these instances Kummer functions appear in the context of the linearisation of the nonlinear operator 
\[
L v:=\Delta v  -\frac{1}{2}\la x,\nabla v\ra -\frac{1}{p-1} v+|v|^{p-1}v    
\]
 around a regular or singular solution of equation \eqref{eq:eq} for $\lambda=1$. In this sense Kummer functions and their special subfamily of Hermite polynomials are structurally associated with our problem. 
 
A general scheme of reasoning could be applied to a larger family of equations e.g.\,we could replace the canonical version of the Ornstein-Uhlenbeck operator with a more general operator and consider equations of the form  
\[
\mathrm{Tr}\left(Q\nabla^2 w\right)+\la P x,\nabla w\ra-a(x)w+f(w)=0, 
\]
where $Q$ is a symmetric positive definite matrix, $P$ is some nonzero matrix, $a(x)$ is a given bounded or singular function and $f$ is of superlinear growth. On the one hand we have taken advantage of some very specific properties of Kummer functions counterparts of which will be unavailable in this generalised situation. On the other hand some of the properties we needed to assert stemmed not so much from the explicit form of the Kummer functions but rather from comparison/maximum/positivity properties of second order equations of a particular structure and which may still be within reach. 

In this paper we have worked in the class of bounded entire solutions and we could ask whether the Liouville Theorem would hold if we weakened regularity assumptions. However, in many cases weak solutions are in fact smooth like for example in the class of locally Lipschitz continuous functions, see \cite{GIGAKOHN1985}, Lemma 1. When it comes to admissible asymptotic behaviour of solutions then (depending on $\lambda$) their boundedness could be relaxed to a suitable polynomial bound on growth. If however we restrict our attention to the class of smooth positive radial solutions then a much stronger statement can be made as according to Lemma 2.1 in \cite{MIZOGUCHI2009} (with obvious modifications accounting for $\lambda\neq 1$) every such solution is bounded and therefore our Liouville Theorem applies to all positive radial entire solutions without any growth bounds. It would be interesting to check whether  the assumption of radial symmetry could be dropped to achieve a level of generality of the celebrated Liouville theorem for positive solutions of the Lane-Emden equation of Gidas and Spruck \cite{GIDASSPRUCK1981}.

Finally, it would be interesting to see whether the Liouville property fails in the critical case outside the range specified in our theorem, as this is the only range where it remains undecided whether nontrivial solutions exist or not. In this branch of analysis of partial differential equations it is often the case that a failure of the Rellich-Pohozaev argument is a strong indicator of the existence of nontrivial solutions. The description of failure of definiteness of the matrix field $\A$ in the proof of Proposition \ref{A+} provides some idea as to the characteristics of possible solutions. In particular for $n=4$ we find $\lambda_\ast=2$, i.e.\,the set of $\lambda$ for which the Liouville theorem holds is given by $(-\infty,1]\cup \{2\}$. It is an intriguing scenario whereby as we increase $\lambda>1$ nontrivial solutions may cease to exist for one value of the parameter only.

\section*{Appendix}
Below we collect a tiny selection of properties of Kummer functions that were used in the preceding calculations. For proofs and further results concerning these remarkable functions an interested reader may consult e.g.\,\cite{BUCHHOLZ,TRICOMI1950}.

 Kummer functions are solutions of Kummer's differential equation
\[
z\frac{\d^2 w}{\d z^2}+(b-z)\frac{\d w}{\d z}-a w=0,
\]
where $a,b,z\in \mathbb C$. In our context however it suffices to restrict our attention to the real case, $z=\xi\in \R$ with $a,b\in \R$ and $b>0$. \\
There exists an analytical solution given by the series  
\begin{equation}\label{A.1}\tag{A.1}
M(a,b,\xi)=1+\frac{a}{b}\xi+\frac{(a)_2}{(b)_2}\frac{\xi^2}{2!}+\dots +\frac{(a)_m}{(b)_m}\frac{\xi^m}{m!}+\dots,
\end{equation}
where 
\[
(a)_m=a(a+1)(a+2)\dots (a+m-1), \quad (a)_0=1.
\]
Observe that when $a$ is a nonpositive integer $M(a,b,\xi)$ reduces to a polynomial. \\
By performing a quadratic substitution, $\xi=\frac{r^2}{4}$, we find that $M\left(a,b,\frac{r^2}{4}\right)$ solves the Hermite equation of the form 
\begin{equation}\label{A.2}\tag{A.2}
u''+\left(\frac{2 b-1}{r}-\frac{r}{2}\right)u'-a u=0 \quad \equiv\quad \left(r^{2 b-1}\rho u'\right)'-a \rho r^{2b-1}u=0 
\end{equation}
for $r>0$ and with $u(0)=1$, 
which in turn is a radial case of the equation 
\begin{equation}\label{A.3}\tag{A.3}
\frac{1}{\rho}\nabla\cdot\left(\rho\nabla u\right)+\frac{2 b-n}{|x|^2}\langle \nabla u,x\rangle-a u=0
\end{equation}
on $\R^n$.

There are three features of Kummer functions that we take advantage of:
\begin{itemize}
\item special identities: 
\begin{equation}\label{A.4}\tag{A.4}\begin{split}
\frac{\partial}{\partial \xi}M(a,b,\xi)&=\frac{a}{b}M(a+1,b+1,\xi),\\
\e^{-\xi}M(a,b,\xi)&=M(b-a,b,-\xi), \\
\frac{a}{b}M(a+1,b+1,\xi)&= M(a,b,\xi)-\frac{b-a}{b}M(a,b+1,\xi), \\
M(a,b,\xi)&=M(a-1,b,\xi)+\frac{\xi}{b} M(a,b+1,\xi),\\
M(b,b,\xi)&=\e^\xi,\\
M(0,b,\xi)&=1;
\end{split}
\end{equation}
\item asymptotic behaviour as $\xi \to +\infty$: 
\begin{equation}\label{A.5}\tag{A.5}
M(a,b,\xi)=\begin{cases}\frac{\Gamma(b)}{\Gamma(a)}\e^\xi \xi^{a-b}\Big[1+\mathcal{O}\left(\xi^{-1}\right)\Big] &\mbox{ when }a\neq 0,-1,-2,\dots, \\
\frac{(a)_{k}}{(b)_{k}}\xi^{k}\Big[1+\mathcal{O}\left(\xi^{-1}\right)\Big] &\mbox{ when } -a=k\in \mathbb{N};
\end{cases}
\end{equation}

\item roots: when $a,b\geq 0$ there are no positive roots of $M(a,b,\xi)$. Let $\pi(a,b)$ be the number of positive zeroes for $M(a,b,\xi)$, then 
\begin{equation}\label{A.6}\tag{A.6}
\pi(a,b)=\begin{cases} 0 &\mbox{ when  }a\geq 0, \, b\geq 0,\\
\left \lceil{-a}\right \rceil &\mbox{ when  }a<0, \, b\geq 0. 
\end{cases}
\end{equation}
Roots are simple so that $M(a,b,\xi)$ changes sign when passing through zero. 
\end{itemize}

\section*{Acknowledgements}
This work was supported by the Polish National Science Center, grant no 2017/26/D/ST1/00614.

\bibliographystyle{plain}
\bibliography{refs}


\end{document}